\documentclass[11pt]{amsart}
\usepackage{graphicx,amssymb,amsmath,amsfonts,amscd,hyperref}
\newtheorem{thm}{Theorem}[section]
\newtheorem{cor}[thm]{Corollary}
\newtheorem{lem}[thm]{Lemma}
\newtheorem{prop}[thm]{Proposition}
\newtheorem{defn}[thm]{Definition}
\newtheorem{rem}[thm]{Remark}

\newcommand{\Trace}{\mathrm{Trace}}

\newcommand{\Sym}{\mathrm{Sym}}
\newcommand{\Hom}{\mathrm{Hom}}

\newcommand{\CC}{\mathbb C}

\newcommand{\QQ}{\bar{\mathbb Q}_\ell}
\newcommand{\R}{\mathrm R}

\newcommand{\AAA}{\mathbb A}
\newcommand{\PP}{\mathbb P}
\newcommand{\FF}{\mathcal F}
\newcommand{\GG}{\mathbb G}
\newcommand{\GGG}{\mathcal G}

\newcommand{\HH}{\mathrm H}
\newcommand{\HHH}{{\mathcal H}}
\newcommand{\LL}{{\mathcal L}}

\newcommand{\Tr}{\mathrm {Tr}}

\newcommand{\Spec}{\mathrm{Spec\:}}
\newcommand{\Dbc}{{\mathcal D}^b_c(\AAA^1_k,\QQ)}

\newcommand{\Frob}{\mathrm{Frob}}
\newcommand{\Swan}{\mathrm{Swan}}

\begin{document}
\title{Improvements of The Weil Bound For Artin-Schreier Curves}
\author{Antonio Rojas-Leon}
\address{Departamanto de \'Algebra,
Universidad de Sevilla, Apdo 1160, 41080 Sevilla, Spain}
\address{E-mail: arojas@us.es}
\thanks{The research of Antonio Rojas-Leon is partially supported by P08-FQM-03894 (Junta de Andaluc\'{\i}a), MTM2007-66929 and FEDER. The research 
of Daqing Wan is partially supported by NSF}

\author{Daqing Wan}
\address{Department of Mathematics, University of California, Irvine, CA 92697-3875, USA}
\address{E-mail: dwan@math.uci.edu}

\begin{abstract}
For the Artin-Schreier curve $y^q  - y  =  f (x)$  defined over a finite field ${\mathbb F}_ q$ of $q$ elements, the 
celebrated Weil bound for the number of ${\mathbb F}_{q^r}$ -rational points can be sharp, especially in super- 
singular cases and when $r$ is divisible. In this paper, we show how the Weil bound can be significantly 
improved, using ideas from moment $L$-functions and Katz's work on $\ell$-adic monodromy calculations. 
Roughly speaking, we show that in favorable 
cases (which happens quite often), one can remove an extra $\sqrt{q}$ factor in the error term.
\end{abstract}

\maketitle

\section{Introduction} Let $k={\mathbb F}_q$ be a finite field of
characteristic $p>2$ with $q$ elements, and let $f\in k[x]$ be a polynomial of degree $d>1$. Without loss of
generality, we can and will always assume that $d$ is not divisible
by $p$. Let $C_f$ be the affine Artin-Schreier curve defined over $k$ by
$$y^q-y = f(x).$$
Let $r$ be a positive integer, and let $N_r(f)$ denote the number of ${\mathbb F}_{q^r}$-rational points on
$C_f$. The genus of the smooth projective model of $C_f$ is given by
$$g =(q-1)(d-1)/2.$$
The celebrated Weil bound in this case gives the estimate
$$|N_r(f) - q^r| \leq (d-1)(q-1)q^{\frac{r}{2}}.$$
This bound can be sharp in general, for instance when $C_f$ is
supersingular and $r$ is divisible. If $q^r$ is not a square, Serre's improvement
\cite{Se} leads to a somewhat better bound:
$$|N_r(f)-q^r| \leq {\frac{(d-1)(q-1)}{2}} [2q^{\frac{r}{2}}],$$
where $[x]$ denotes the integer part of a real number $x$.

In this paper, we shall show that if $q$ is large compared to $d$
(and thus the genus $g=(d-1)(q-1)/2$ is small compared to the field
size $q^r$ with $r\geq 2$), then the above Weil bound can be
significantly improved in many cases. The type of theorems we prove
is of the following nature. For simplicity, we just state one special
case.

\begin{thm}\label{introduction}Let $r\geq 1$ and $p>2$. If the derivative
$f'$ is square-free  
and either $r$ is odd or the hypersurface $f(x_1)+\cdots+f(x_r)=0$ in $\AAA^r_k$ is non-singular, then we
have the estimate
$$|N_r(f) - q^r| \leq C_{d,r}q^{\frac{r+1}{2}},$$
where $C_{d,r}$ is the constant
$$C_{d,r} = \sum_{a=0}^r |a-1|{d-2+r-a\choose r-a}{d-1\choose a}.$$
\end{thm}
Note that the constant $C_{d,r}$ is independent of $q$ and it is a
polynomial in $d$ with degree $r$. Thus, for fixed $d$ and $r$, our
result essentially removes an extra $\sqrt{q}$ factor from Weil's
bound. The non-singularity hypothesis cannot be dropped in general,
as there are cases for $r$ even where we can have
$$|N_r(f) - (q^r + q^{\frac{r}{2}+1})| \leq C_{d,r} q^{\frac{r+1}{2}},$$
see section \ref{refinements} for more details. This gives further examples that the
$q$-factor in the Weil bound cannot be replaced by an $O(\sqrt{q})$
factor in general.

As an extreme illustration, we consider the elementary case that
$r=1$. It is clear that $N_1(f)=qn_f$, where $n_f$ is the number of
distinct roots of $f(x)$ in ${\mathbb F}_q$ which is at most $d$.
Thus, the best estimate in this case should be
$$|N_1(f)-q | \leq (d-1)q,$$
which is precisely what our bound gives!  It is far better than the
Weil bound
$$|N_1(f)-q| \leq (d-1)(q-1)\sqrt{q}.$$
For $r=2$, our bound takes the form
$$|N_2(f)-q^2| \leq (d-1)^2q^{3/2},$$
which is better than the Weil bound
$$|N_2(f)- q^2 | \leq (d-1)(q-1)q$$
as soon as $q\geq (d-1)^2+3$.
For $r=3$, our bound takes the form
$$|N_3(f)-q^3| \leq (d-1)(d^2-3d+3)q^2,$$
which is better than the Weil bound
$$|N_3(f)- q^3 | \leq (d-1)(q-1)q^{3/2}$$
as soon as $q\geq (d^2-3d+4)^2$.

Our idea is to translate $N_r(f)$ to moment exponential sums and
then calculate the associated moment $L$-function as explicitly as
possible. Let $\psi$ be a fixed non-trivial additive character of
$k$. For $f\in k[x]$, it is clear that we
have the formula
$$
N_r(f) = \sum_{t\in k}\sum_{x\in k_r}
\psi({\rm Tr}(tf(x))),
$$
where $k_r={\mathbb F}_{q^r}$ and ${\rm Tr}$ denotes the trace map from $k_r$ to $k$. Separating the term from $t=0$, we obtain
\begin{equation}\label{charactersum}
N_r(f)-q^r = \sum_{t\in k^\star}\sum_{x\in k_r}
\psi({\rm Tr}(tf(x))).
\end{equation}
Now, Weil's bound for exponential sums gives the estimate
$$
\left|\sum_{x\in k_r }\psi({\rm Tr}(tf(x)))\right| \leq (d-1)q^{\frac{r}{2}}
$$
for every $t\in k^\star$. It follows that
$$|N_r(f)-q^r| \leq (q-1)(d-1)q^{\frac{r}{2}}.$$
In order to improve this bound, we need to understand the
cancelation of the outer sum of (\ref{charactersum}) over $t\in k^\star$.
Heuristically, one expects that the outer sum contributes another
$O(\sqrt{q})$ factor instead of the trivial $q$ factor, if $f$ is
sufficiently ``random". This is in fact what we shall prove using
the full strength of Deligne's general theorem on Riemann
hypothesis.

The double sum in (\ref{charactersum}) is precisely a moment exponential sum
associated to the two variable polynomial $tf(x)$. Thus, we can use
the techniques of moment $L$-functions to get improved information
about the solution number $N_r(f)$. We now briefly outline our
method. Let $\ell$ be a fixed prime different from $p$. Let $\GGG_f$
denote the relative $\ell$-adic cohomology with compact support
associated to the family of one variable exponential sums attached to
$tf(x)$, where $x$ is the variable and $t$ is the parameter on the
torus ${\mathbb G}_m$. Applying the $\ell$-adic trace formula fibre
by fibre, we obtain
$$\sum_{t\in k^\star}\sum_{x\in k_r }\psi({\rm Tr}(tf(x))) =-\sum_{t\in
k^\star } {\rm Tr}(\Frob_q^r|(\GGG_f)_t),$$ where $(\GGG_f)_t$
is the fibre of $\GGG_f$ at $t$, and ${\rm Frob}_q$ is the geometric
$q$-th power Frobenius map. Alternatively, one can rewrite
$${\rm Tr}({\rm Frob}_q^r|(\GGG_f)_t) = {\rm Tr}({\rm Frob}_q | [\GGG_f]^r_t),$$
where $[\GGG_f]^r$ denotes the $r$-th Adams operation of $\GGG_f$. It is a
virtual $\ell$-adic sheaf on ${\mathbb G}_m$. For example, Katz \cite{KatzFS} used the 
formula 
 $$[\GGG_f]^r =   \sum_{i=1}^r (-1)^{i-1} i \cdot {\rm Sym}^{r-i}\GGG_f \otimes \wedge^i
\GGG_f.$$ 
We shall use the following optimal formula from \cite{Wan} given by 
$$[\GGG_f]^r =\sum_{i=0}^r (-1)^{i-1} (i -1)\cdot {\rm Sym}^{r-i}\GGG_f \otimes \wedge^i
\GGG_f.$$ 
Note that the term $i=0$ does not occur in the first formula, and the term $i=1$ does not 
occur in the second formula as the coefficient becomes zero for $i=1$.  
The coefficients of the second formula are smaller and thus lead to fewer number of zeros and 
poles for the corresponding $L$-functions. In this way, we get the smaller constant $C_{d,r}$ 
in Theorem 1.1. 

It follows that
$$N_r(f)-q^r =\sum_{i=0}^r (-1)^{i} (i -1)\cdot \sum_{t\in k^\star} {\rm Tr}(\Frob_q |
({\rm Sym}^{r-i}\GGG_f \otimes \wedge^i \GGG_f)_t).$$ This reduces
our problem to the study of the $L$-function over ${\mathbb G}_m$ of
the $\ell$-adic sheaves ${\rm Sym}^{r-i}\GGG_f \otimes \wedge^i
\GGG_f$ for all $0\leq i\leq r$. By general results of Deligne
\cite{De}, we deduce that
$$|N_r(f)- ( q^r +\delta_{f,r} q^{\frac{r}{2}+1})| \leq C_{d,r} q^{\frac{r+1}{2}},$$
where $C_{d,r}$ comes from the Euler characteristic of the components of the virtual
sheaf $[\GGG_f]^r$, and
$$\delta_{f,r} =\sum_{i=0}^r (-1)^{i-1} (i-1) \cdot {\rm dim} \HH_c^2({\mathbb G}_{m,\bar k},
{\rm Sym}^{r-i}\GGG_f \otimes \wedge^i \GGG_f).$$
 Under the conditions of Theorem \ref{introduction}, it follows that the sheaf ${\rm Sym}^{r-i}\GGG_f \otimes
\wedge^i \GGG_f$ has no geometrically trivial component for any $0\leq
i\leq r$, and thus we deduce that $\delta_{f,r}=0$.

Our main result is somewhat stronger. We determine the weights, the trivial factors and the
degrees of the $L$-functions of all the sheaves ${\rm Sym}^{r-i}\GGG_f
\otimes \wedge^i \GGG_f$, thus obtaining a fairly
complete information about the associated moment $L$-function, see
\cite{FW1}\cite{FW2} and \cite{AW} for the study of moment
$L$-functions in two other examples, namely, the family of
hyper-Kloosterman sums and the Dwork family of toric Calabi-Yau
hypersurfaces. Under slightly more general hypotheses, Katz's results on monodromy group
calculations \cite{monodromy}\cite{esde} give stronger results which lead to further improvements
of Theorem \ref{introduction} (see Corollaries \ref{estimateSL} and \ref{estimateSp}). See also \cite{KatzFS} for a result on the average number of rational points on hypersurfaces obtained using a similar approach.

The possibility of our improvement for the Weil bound in the case of
Artin-Schreier curves is due to the fact that the curve has a large
automorphism group ${\mathbb F}_q$, which is the group of ${\mathbb
F}_q$-rational points on the group scheme ${\mathbb A}^1$. We expect
that similar improvements should exist for many other curves (or
higher dimensional varieties) with a large automorphism group. For
example, in the last section of this paper we treat the case
of Artin-Schreier hypersurfaces
$$y^q-y =f(x_1,..., x_n).$$
This method leads to similar improvements of Deligne's bound for
such hypersurfaces in many cases. As an explicit new example to try,
we would suggest the affine Kummer curve of the form
$$y^{\frac{(q-1)}{e}} = f(x),$$
where $e$ is a fixed positive integer, $q$ is a prime power congruent to
$1$ modulo $e$, and $f(x)\in k[x]$ is a polynomial of
degree $d$. For $r\geq 1$ and $N_r(f, e)$ denoting the number of
${\mathbb F}_{q^r}$-rational points on the above Kummer curve, we
conjecture that for certain generic $f$, there is the following
estimate
$$|N_r(f,e) -q^r| \leq s_{d,e,r} q^{\frac{r+1}{2}},$$
where $s_{d,e,r}$ is a constant independent of $q$. We do not
know how to prove this conjecture, even in the case $e=1$.

To conclude this introduction, we raise another open problem. In
Theorem 1, we assumed that the curve $C_f: y^q-y=f(x)$ is defined
over the subfield ${\mathbb F}_q$ of ${\mathbb F}_{q^r}$. We believe
that similar improvement is also true if $C_f$ is defined over the
larger field ${\mathbb F}_{q^r}$. But we could not prove this at
present.

{\bf Remarks}. Weil's estimate gives both an upper bound and a lower
bound for the number of rational points on a curve of genus $g$ over
the finite field ${\mathbb F}_q$. Improvements for the lower bound
are in general harder to get. Improvements for the upper bound can
often be obtained by more elementary means. In fact, there are
already several such results in the literature for large genus
curves. The first result along these lines is due to Stark \cite{St}
in the hyperelliptic case, using Stepanov's method. Using the
explicit formula, Drinfeld-Vladut and Serre \cite{Se} obtained an
upper bound improvement when $2g> q^r -q^{r/2}$, which in our
Artin-Schreier setting becomes
$$(d-1)(q-1) > q^r -q^{r/2}.$$
For $r>1$, this means that $q$ must be small compared to $d$. In
comparison, our improvements apply when $q$ is large compared to
$d$. Using a geometric intersection argument, St\"oher-Voloch
\cite{SV} obtained another upper bound which in our case becomes
$$N_r(f) \leq {1\over 2}D(D+q^r-1),$$
where $D=\max(d, q)$.

{\bf Acknowledgment}. It is a pleasure to thank the referee for his careful reading of the first version and for his   
very helpful comments.

\section{Cohomology of the family $t\mapsto \sum\psi(\Tr(tf(x)))$}

Let $k={\mathbb F}_q$ be a finite field of characteristic $p$, and $f\in k[x]$ a
polynomial of degree $d$ prime to $p$. Let $C_f$ be the Artin-Schreier curve defined on $\AAA^2_k$ by the equation
\begin{equation}
 y^q-y=f(x)
\end{equation}
and denote by $N_r(f)$ its number of rational points over $k_r:={\mathbb F}_{q^r}$.

Fix a non-trivial additive character $\psi:k\to \CC^\star$. It is clear that
\begin{equation}
N_r(f)=\sum_{t\in k}\sum_{x\in k_r}\psi(t\cdot\Tr(f(x)))=\sum_{t\in k}\sum_{x\in k_r}\psi(\Tr(tf(x)))
\end{equation}
where $\Tr$ denotes the trace map $k_r\to k$.

Fix a prime $\ell\neq p$ and an isomorphism $\iota:\QQ\to{\mathbb C}$. Consider
the Galois \'etale cover of ${\mathbb G}_m\times{\mathbb A}^1$ (with
coordinates $(t,x)$) given by $u-u^q=tf(x)$, with Galois group $k$;
and let ${\mathcal L}_{\psi(tf(x))}$ be the rank 1 smooth
$\QQ$-sheaf corresponding to the representation of $k$ given by
$\psi^{-1}$ via $\iota$. Define $K_f=\R\pi_!{\mathcal L}_{\psi(tf(x))}\in{{\mathcal
D}^b_c({\mathbb G}_{m,k},\QQ)}$, where $\pi:{\mathbb
G}_m\times{\mathbb A}^1\to{\mathbb G}_m$ is the projection. The
trace formula implies that the trace of the action of the $r$-th power of a local geometric Frobenius element 
at $t\in k^\star$ on $K_f$ is given by $\sum_{x\in k_r}\psi(\Tr(tf(x)))$.

It is known \cite[3.7]{De} that $K_f={\mathcal G}_f[-1]$ for a
smooth sheaf ${\mathcal G}_f$ of rank $d-1$ and punctually pure of
weight $1$, whose local $r$-th power Frobenius trace at $t\in k^\star$ is then
given by $-\sum_{x\in k_r}\psi(\Tr(tf(x)))$. Therefore
\begin{equation}\label{nr}
N_r(f)-q^r=\sum_{t\in k^\star}\sum_{x\in k_r}\psi(\Tr(tf(x)))\end{equation}
$$=
-\sum_{t\in k^\star}\Tr({\rm Frob}^r_t|(\GGG_f)_t) =  -\sum_{t\in
k^\star}{\rm Tr}({\rm Frob}_t | [\GGG_f]^r_t)
$$
where
$$
[\GGG_f]^r =\sum_{i=0}^r (-1)^{i-1} (i -1)\cdot {\rm Sym}^{r-i}\GGG_f \otimes \wedge^i
\GGG_f
$$
is the $r$-th Adams operation on $\GGG_f$.

The sheaf ${\mathcal G}_f$ can also be interpreted in terms of the
Fourier transform. Consider the sheaf $f_\star\QQ$ on ${\mathbb
A}^1_k$. There is a canonical surjective trace map
$\phi:f_\star\QQ=f_\star f^\star\QQ\to\QQ$, let ${\mathcal F}_f$ be
its kernel. It is a constructible sheaf of generic rank $d-1$ on
${\mathbb A}^1_k$.

\begin{lem}\label{fourier} If $j:{\mathbb G}_{m,k}\to{\mathbb A}^1_k$ is the inclusion,
the shifted sheaf $j_!{\mathcal G}_f[1]$ is the Fourier transform of ${\mathcal F}_f[1]$ with
respect to $\psi$.\end{lem}

{\bf Proof.} Taking Fourier transform in the distinguished triangle
in $\Dbc$:
$$
{\mathcal F}_f[1]\to f_\star\QQ[1]\to \QQ[1]\to
$$
we get a distinguished triangle:
$$
FT_\psi({\mathcal F}_f)[1]\to
FT_\psi(f_\star\QQ)[1]\to(\QQ)_0(-1)[0]\to.
$$
where $(\QQ)_0$ is a punctual sheaf supported at $0$. If
$\mu:\AAA^1\times\AAA^1\to\AAA^1$ is the multiplication map, the
Fourier transform of $f_\star\QQ[1]$ is given by
$\R{\pi_1}_!(\pi_2^\star f_\star\QQ\otimes\mu^\star{\mathcal
L}_{\psi})[2]=\R{\pi_1}_!({\mathcal L}_{\psi(tf(x))})[2]$, where
$\pi_i:\AAA^1\times\AAA^1\to\AAA^1$ are the projections. In particular, by proper base change $j^\star FT_\psi(f_\star\QQ)[1]=j^\star\R{\pi_1}_!({\mathcal L}_{\psi(tf(x))})[2]=K_f[2]=\GGG_f[1]$. Applying $j^\star$ to the triangle above we find quasi-isomorphisms
$$
j^\star FT_\psi(\FF_f)[1]\cong\GGG_f[1]
$$
and 
$$
j_!j^\star FT_\psi(\FF_f)[1]\cong j_!\GGG_f[1].
$$

To conclude, it remains to show that the natural map $j_!j^\star FT_\psi(\FF_f)[1]\to FT_\psi(\FF_f)[1]$ is a quasi-isomorphism. Since its restriction to $\GG_{m,k}$ is a quasi-isomorphism, we only need to check that it induces a quasi-isomorphism on the stalks at (a geometric point over) $0$, that is, that $FT_\psi(\FF_f)_0=0$. By definition of the Fourier transform, $FT_\psi(\FF_f)_0=\R\Gamma_c(\AAA^1_{\bar k},\FF_f)$. We conclude by using the long exact sequence of cohomology with compact support associated to the sequence
$$
0\to\FF_f\to f_\star\QQ\to\QQ\to 0,
$$
since $\HH^i_c(\AAA^1_{\bar k},f_\star\QQ)=\HH^i_c(\AAA^1_{\bar k},\QQ)=0$ for $i\neq 2$ and $\HH^2_c(\AAA^1_{\bar k},f_\star\QQ)=\HH^2_c(\AAA^1_{\bar k},\QQ)=\QQ(-1)$ is one-dimensional.
\hfill$\Box$

\bigskip

We can now use Laumon's local Fourier transform theory to determine
the monodromy actions at $0$ and $\infty$ for ${\mathcal G}_f$. Recall that, for every character $\chi:k^\star\to\QQ^\star$, there is an associated Kummer sheaf $\LL_\chi$ on $\GG_{m,k}$: The $(q-1)$-th power map $\GG_{m,k}\to\GG_{m,k}$ is a Galois \'etale cover with Galois group canonically isomorphic to $k^\star$, and one just takes the pull-back of the character $\bar\chi$ to $\pi_1(\GG_{m,k},\bar\eta)\twoheadrightarrow k^\star$. For every $d|q-1$, if $[d]$ denotes the $d$-th power map $\GG_{m,k}\to\GG_{m,k}$ we have $[d]_\star\QQ=\bigoplus\LL_\chi$, where the sum is taken over all characters of $k^\star$ such that $\chi^d$ is trivial.

Assume that $k$ contains all $d$-th roots of unity. The sheaf $\FF_f$ is smooth on the complement $U$ of the set of the critical values of $f$ in $\AAA^1$. Since $d$ is prime to $p$, in a neighborhood of infinity the map $x\mapsto f(x)=a_dx^d(1+\frac{a_{d-1}}{a_dx}+\cdots+\frac{a_0}{a_dx^d})$ is equivalent (for the \'etale topology) to the map $x\mapsto a_d x^d$ (just by making the change of variable $x\mapsto \alpha x$, where $\alpha^d=1+\frac{a_{d-1}}{a_dx}+\cdots+\frac{a_0}{a_dx^d})$. In particular, the decomposition group $D_\infty$ at infinity acts on the generic stalk of $f_\star\QQ$ through the direct sum of the tame characters $(a_d)_\star\LL_\chi$ for all non-trivial characters $\chi$ of $k^\star$ such that $\chi^d={\mathbf 1}$, where $(a_d):\GG_{m,k}\to\GG_{m,k}$ is the multiplication by $a_d$ map. Since $(a_d)_\star\LL_\chi=(a_d^{-1})^\star\LL_\chi=\bar\chi(a_d)^{deg}\otimes\LL_\chi$, we conclude that $D_\infty$ acts on the generic stalk of $\FF_f$ via the direct sum $\bigoplus\bar\chi(a_d)^{deg}\otimes\LL_\chi$ taken over all non-trivial characters $\chi$ of $k^\star$ such that $\chi^d$ is trivial.

\begin{prop}\label{monodromy0} Suppose that $k$ contains all $d$-th roots of unity. The action of the decomposition group $D_0$ at $0$ on ${\mathcal G}_f$ is
tame and semisimple, and it splits as a direct sum $\bigoplus(\chi(a_d)g(\bar\chi,\psi))^{deg}\otimes\LL_\chi$ over
all non-trivial characters $\chi$ of $k^\star$ such that $\chi^d={\mathbf 1}$, where $g(\bar\chi,\psi):=-\sum_t\bar\chi(t)\psi(t)$ is the Gauss sum.\end{prop}

{\bf Proof.} By (\cite[Proposition 2.5.3.1]{laumon},\cite[Theorem 7.5.4]{esde}), the local monodromy at $0$ of ${\mathcal
G}_f$ can be read from the local monodromy at infinity of ${\mathcal
F}_f$. More precisely, we have $LFT^{(\infty,0)}(\bigoplus\bar\chi(a_d)^{deg}\otimes\LL_\chi)=\bigoplus LFT^{(\infty,0)}(\bar\chi(a_d)^{deg}\otimes\LL_\chi)$. Now, for every $\chi$, since the Fourier transform commutes with tensoring by an unramified sheaf (by the projection formula, since $\pi_1^\star(\alpha^{deg})=\alpha^{deg}$ and $\mu^\star(\alpha^{deg})=\alpha^{deg}$ for $\pi_1$ and $\mu:\AAA^1_k\times\AAA^1_k\to\AAA^1_k$ the projection and multiplication) we have $LFT^{(\infty,0)}(\bar\chi(a_d)^{deg}\otimes\LL_\chi)=\bar\chi(a_d)^{deg}\otimes LFT^{(\infty,0)}\LL_\chi=\bar\chi(a_d)^{deg}\otimes g(\chi,\psi)^{deg}\otimes\LL_{\bar\chi}$ by \cite[Proposition 2.5.3.1]{laumon} (note that $\LL_\chi$ corresponds to $V_{\bar\chi}$ as a representation of $D_\infty$ and to $V'_\chi$ as a representation of $D_0$ in the notation of \cite{laumon} due to the choice of uniformizers). \hfill$\Box$

For simplicity, we will assume from now on that $f'$ is square-free and $p>2$. Suppose that $k$ contains all roots of $f'$ (and therefore all critical values of $f$). Let $s\in k$ be a critical value of $f$. The polynomial $f_s:=f-s$ has at worst double roots and $k$ contains all its double roots. Let $g_s$ be the square-free part of $f_s$ (i.e. $f_s$ divided by the product of all its monic double linear factors), which lies in $k[x]$.
Let $S_0$ be the henselization of $\AAA^1_k$ at $s$, $z_1,\ldots,z_e\in k$ the double roots of $f_s$, $S_j$ the henselization of $\AAA^1_k$ at $z_j$ for $j=1,\ldots,e$ and $T$ the union of the henselizations of $\AAA^1_k$ at the closed points of the subscheme defined by $g_s=0$. We have a cartesian diagram
$$
\begin{CD}
( \coprod_j S_j)\coprod T @>>> \AAA^1_k   \\
@VV (\coprod_j h_j)\coprod h V @VVfV \\
S_0 @>>> \AAA^1_k
\end{CD}
$$
where the map $h_j:S_j\to S_0$ is isomorphic (for the \'etale topology) to the map $x\mapsto b_j(x-z_j)^2$ (where $b_j$ is $f_s(x)/(x-z_j)^2$ evaluated at $z_j$, that is, $f''(z_j)/2$) via the change of variable mapping the local coordinate $x-z_j$ to $\alpha(x-z_j)$, where $\alpha\in S_j$ is a square root of $f_s(x)/b_j(x-z_j)^2$ (which exists by Hensel's lemma, since its image in the residue field $k$ is $1$), and $h:T\to S_0$ is finite \'etale. In particular, the decomposition group $D_s$ at $s$ acts on the generic stalk of $f_\star\QQ$ through the direct sum $\bigoplus_j ({\mathbf 1}\oplus(b_j)_\star\LL_\rho)\bigoplus L=\bigoplus_j(\rho(b_j)^{deg}\otimes\LL_\rho)\bigoplus(e\cdot{\mathbf 1}\oplus L)$ where $L$ is unramified and $\rho=\bar\rho:k^\star\to\QQ^\star$ is the quadratic character.

\begin{prop}\label{monodromyinf} Suppose that $p>2$, $f'$ is square-free and all its roots are in $k$. The action of the decomposition group $D_\infty$
at infinity on ${\mathcal G}_f$ splits as a direct sum
$\bigoplus_z(\rho(b_z)g(\rho,\psi))^{deg}\otimes\LL_\rho\otimes\LL_{\psi_{f(z)}}$ where the sum is taken over the roots of $f'$, $b_z=f''(z)/2$, $\rho:k^\star\to\QQ^\star$ is the quadratic character and $g(\rho,\psi)=-\sum_t\rho(t)\psi(t)$ the corresponding Gauss sum.
\end{prop}

{\bf Proof.} By (\cite{laumon},\cite[Theorem 7.5.4]{esde}), the local monodromy at infinity of ${\mathcal
G}_f$ can be read from the local monodromies of ${\mathcal
F}_f$. More precisely, the part of slope $>1$ corresponds to the slope $>1$ part of the local monodromy at infinity of $\FF_f$, so it vanishes. The part of slope $\leq 1$ is a direct sum, over all critical values $s$ of $f$, of ${\mathcal L}_{\psi_{s}}$ tensored with the local Fourier transform $LFT^{(0,\infty)}$ applied to the action of $I_s$ on the generic stalk of $\FF_f$ modulo its $I_s$-invariant space. 

Using \cite[2.5.3.1]{laumon} and the fact that Fourier transform commutes with tensoring by unramified sheaves, for every root $z$ of $f'$ $LFT^{(0,\infty)}(\rho(b_z)^{deg}\otimes\LL_\rho)=\rho(b_z)^{deg}\otimes g(\rho,\psi)^{deg}\otimes\LL_\rho$. So each critical value $s$ contributes a factor $\bigoplus_{f(z)=s}(\rho(b_z)g(\rho,\psi))^{deg}\otimes\LL_\rho\otimes\LL_{\psi_s}$ to the monodromy of $\GGG_f$ at infinity.
\hfill$\Box$

\bigskip

We can now compute the determinant of ${\mathcal G}_f$:

\begin{cor}\label{determinant} Suppose that $k$ contains all $d$-th roots of unity.
If $d$ is odd, the determinant of ${\mathcal G}_f$ is the Tate-twisted Artin-Schreier
sheaf ${\mathcal L}_{\psi_s}((1-d)/2)$, where $s=s_1+\cdots+s_{d-1}$ is the sum of the critical values of $f$ 
and $\psi_s(t)=\psi(s t)$. If $d$ is even, the determinant of
${\mathcal G}_f$ is
${\mathcal L}_\rho\otimes{\mathcal L}_{\psi_s}\otimes(\epsilon\rho(a_d)g(\rho,\psi))^\mathrm{deg}((2-d)/2)$,
where $\rho$ is the multiplicative character of order $2$, $g(\rho,\psi)=-\sum_t\rho(t)\psi(t)$
is the corresponding Gauss sum, $a_d$ is the leading coefficient of $f$, $\epsilon=1$
if $d\equiv 0$ or $2\mod 8$ and $\epsilon=(-1)^{(q-1)/d}$ if $d\equiv 4$ or $6\mod 8$.
\end{cor}

{\bf Proof.} The determinant of ${\mathcal G}_f$ is a smooth sheaf
of rank one on $\GG_{m,k}$. At $0$, it is isomorphic by proposition \ref{monodromy0} to the product $\bigotimes_{\chi^d={\mathbf 1}, \chi\not={\mathbf 1}} \left((\chi(a_d)g(\bar\chi,\psi))^{deg}\otimes{\mathcal
L}_\chi\right)$. For any $\chi$ we have $\left((\chi(a_d)g(\bar\chi,\psi))^{deg}\otimes{\mathcal
L}_\chi\right)\otimes\left((\bar\chi(a_d)g(\chi,\psi))^{deg}\otimes{\mathcal
L}_{\bar\chi}\right)=(g(\bar\chi,\psi)g(\chi,\psi))^{deg}=(\chi(-1)q)^{deg}$. If $d$ is odd, the non-trivial characters with $\chi^d= {\mathbf 1}$ can be grouped in conjugate pairs. Moreover, $\chi(-1)=\chi((-1)^d)=\chi^d(-1)=1$. We conclude that the determinant at $0$ is the unramified character $(q^{\frac{d-1}{2}})^{deg}=\QQ(\frac{1-d}{2})$. At infinity, it is geometrically isomorphic by proposition \ref{monodromyinf} to the
product $\bigotimes_z(\LL_\rho\otimes\LL_{\psi_{f(z)}})=\LL_{\psi_s}$ (the hypothesis that $k$ contains all roots of $f'$ is not needed for the \emph{geometric} isomorphism, since it is always satisfied in a sufficiently large finite extension of $k$). So $\det(\GGG_f)\otimes\LL_{\psi_{-s}}$ is everywhere unramified and therefore geometrically constant. Looking at the Frobenius action at $0$, it must be $\QQ(\frac{1-d}{2})$, so $\det(\GGG_f)=\LL_{\psi_s}(\frac{1-d}{2})$.

If $d$ is even, the factor at $0$ corresponding to the quadratic character $\rho$ stays unmatched, so as a representation of $D_0$ the determinant is $(\epsilon q^{\frac{d-2}{2}})^{deg}\otimes(\rho(a_d)g(\rho,\psi))^{deg}\otimes\LL_\rho$, where $\epsilon=\prod_{i=1}^{(d-2)/2}\chi^i(-1)$ for a fixed character $\chi$ of exact order $d$. At $\infty$ it is geometrically isomorphic to $\LL_\rho\otimes\LL_{\psi_s}$, so $\det(\GGG_f)\otimes\LL_\rho\otimes\LL_{\psi_{-s}}$ is everywhere unramified and therefore geometrically constant. Looking at the Frobenius action at $0$, it must be $(\epsilon q^{\frac{d-2}{2}}\rho(a_d)g(\rho,\psi))^{deg}$, so $\det(\GGG_f)=\LL_{\rho}\otimes\LL_{\psi_s}\otimes(\epsilon\rho(a_d)g(\rho,\psi))^{deg}(\frac{2-d}{2})$.

It remains to compute the value of $\epsilon$. We have 
$$\epsilon=\prod_{i=1}^{(d-2)/2}\chi^i(-1)=\chi^{d(d-2)/8}(-1)=\chi((-1)^{d(d-2)/8}).$$
If $d\equiv 0$ or $2\mod 8$, $d(d-2)/8$ is even and therefore
$\epsilon=1$. If $d\equiv 4$ or $6\mod 8$, $d(d-2)/8$ is odd so
$\epsilon=\chi(-1)=(-1)^{(q-1)/d}$.\hfill$\Box$

\section{The moment $L$-function of ${\mathcal G}_f$.}

 Recall the definition \cite{FW3} of the moment $L$-function for the sheaf ${\mathcal G}_f$. For a fixed $r\geq 1$, let
$$
L^r(f,\psi,T):=\prod_{t\in |\GG_{m,k}|}{1\over \det(1-\mathrm{Frob}^r_tT^{\deg(t)}|(\GGG_f)_t)},
$$
where $|\GG_{m,k}|$ denotes the set of closed points of $\GG_{m,k}$. 

It is known (\cite[Theorem 1.1]{FW3}) that $L^r(f,\psi,T)$ is a rational function, and we have the formula
\begin{equation}\label{lfunction}
L^r(f,\psi,T)=\frac{\det(1-\mathrm{Frob}_kT|\HH^1_c(\GG_{m,\bar k},[\GGG_f]^r))}{\det(1-\mathrm{Frob}_kT|\HH^2_c(\GG_{m,\bar k},[\GGG_f]^r))}=
\end{equation}
$$
=\frac{\prod_{i=0}^r \det(1-\mathrm{Frob}_kT|\HH^1_c(\GG_{m,\bar k},
\Sym^{r-i}\GGG_f\otimes\wedge^i\GGG_f))^{(-1)^{i-1}(i-1)}}{\prod_{i=0}^r
\det(1-\mathrm{Frob}_kT|\HH^2_c(\GG_{m,\bar
k},\Sym^{r-i}\GGG_f\otimes\wedge^i\GGG_f))^{(-1)^{i-1}(i-1)}}.
$$
Thus, we get a decomposition

\begin{equation}\label{lfunctiondecomposition}
L^r(f,\psi,T)={Q(T)P_0(T)P_\infty(T)\over P(T)P'(T)}.
\end{equation}
We now describe each of the factors in this decomposition.

First,
$$
Q(T)=\prod_{i=0}^r \det(1-\mathrm{Frob}_kT|\HH^1(\PP^1_{\bar k},j_\star(\Sym^{r-i}\GGG_f\otimes\wedge^i\GGG_f)))^{(-1)^{i-1}(i-1)}
$$
is the non-trivial factor. Notice that the dual of $\GGG_f$ is $\GGG_{-f}(1)$, since $\FF_f$ is self-dual and $D\circ FT_\psi=FT_{\bar\psi}\circ D (1)$ \cite[Corollaire 2.1.5]{katz-laumon} and $FT_{\bar\psi}\FF_f[1]=[t\mapsto -t]^\star\GGG_{f}[1]=\GGG_{-f}[1]$. Therefore the dual of $\Sym^{r-i}\GGG_f\otimes\wedge^i\GGG_f$ is $\Sym^{r-i}\GGG_{-f}\otimes\wedge^i\GGG_{-f}(r)$, so the dual (in the derived category) of $j_\star\Sym^{r-i}\GGG_f\otimes\wedge^i\GGG_f[1]$ is $j_\star\Sym^{r-i}\GGG_{-f}\otimes\wedge^i\GGG_{-f}[1](r+1)$, cf. \cite[2.1]{dualite}. Since $\PP^1$ is proper, by \cite[Th\'eor\`eme 2.2]{dualite} we get a perfect pairing
$$
\HH^1(\PP^1_{\bar k},j_\star(\Sym^{r-i}\GGG_f\otimes\wedge^i\GGG_f))\times\HH^1(\PP^1_{\bar k},
j_\star(\Sym^{r-i}\GGG_{-f}\otimes\wedge^i\GGG_{-f}))\mapsto \QQ(-r-1)
$$
for every $i=0,\ldots,r$.

 In particular,
we get a functional equation relating the polynomial
$$Q_i(T):={\det(1-\mathrm{Frob}_kT|\HH^1(\PP^1_{\bar k},
j_\star(\Sym^{r-i}\GGG_f\otimes\wedge^i\GGG_f)))}=\prod_{j=1}^{s_i}(1-\gamma_{ij}
T).$$
and the corresponding polynomial $Q_i^\star(T)$ for $-f$. The functional equation is given by 
$$
{Q_i^\star(T)}=\prod_{j=1}^{s_i}(1-q^{r+1}\gamma_{ij}^{-1}T)=
$$
$$
=\frac{T^{s_i}q^{(r+1){s_i}}}{(-1)^{s_i}\gamma_{i1}\cdots\gamma_{is_i}}
\prod_{j=1}^{s_i}(1-\gamma_{ij}q^{-(r+1)}T^{-1})=\frac{T^{s_i}q^{(r+1){s_i}}}{c_{s_i}}Q_i(q^{-(r+1)}T^{-1})
$$
where $c_{s_i}$ is the leading coefficient of $Q_i(T)$. Therefore,
$$
{Q^\star(T)}:=\prod_{i=0}^r Q^\star_i(T)^{(-1)^{i-1}(i-1)}=\frac{T^sq^{(r+1)s}}{c_s}Q(q^{-(r+1)}T^{-1})
$$
where $s$ is the degree of the rational function $Q(T)$ and $c_s$ its leading coefficient (i.e. the ratio of the leading coefficients of the numerator and denominator). By \cite[Th\'eor\`eme 3.2.3]{De}, all reciprocal roots and poles of $Q(T)$ are pure Weil integers of weight $r+1$.

The other factors of $L^r(f,\psi,T)$ are the ``trivial factors":
$$
P(T)=\prod_{i=0}^r \det(1-\mathrm{Frob}_kT|\HH^0(\PP^1_{\bar k},j_\star(\Sym^{r-i}\GGG_f\otimes\wedge^i\GGG_f)))^{(-1)^{i-1}(i-1)}
$$
and
$$
P'(T)=\prod_{i=0}^r \det(1-\mathrm{Frob}_kT|\HH^2(\PP^1_{\bar k},j_\star(\Sym^{r-i}\GGG_f\otimes\wedge^i\GGG_f)))^{(-1)^{i-1}(i-1)}
$$
are rational functions of the same degree and pure of weight $r$ and $r+2$ respectively, and vanish if $\Sym^{r-i}\GGG_f\otimes\wedge^i\GGG_f$ has no invariants for the action of $\pi_1(\GG_{m,\bar k})$ for any $i$. The other two are the local factors at $0$:
$$
P_0(T):=\det(1-\Frob_0T|([{\mathcal G}_f]^{r})^{I_0})=
$$
$$
=\prod_{i=0}^r \det(1-\mathrm{Frob}_0T|(\Sym^{r-i}\GGG_f\otimes\wedge^i\GGG_f)^{I_0})^{(-1)^{i-1}(i-1)}
$$
and at infinity:
$$
P_\infty(T):=\det(1-\Frob_\infty T|([{\mathcal G}_f]^{r})^{I_\infty})=
$$
$$
=\prod_{i=0}^r \det(1-\mathrm{Frob}_\infty T|(\Sym^{r-i}\GGG_f\otimes\wedge^i\GGG_f)^{I_\infty})^{(-1)^{i-1}(i-1)}.
$$
We now compute the local factors explicitly. 

\begin{cor} \label{local0}\emph{(Local factor at $0$ of the moment $L$-function)}
Suppose that $k$ contains all $d$-th roots of unity. For any positive integer $r\geq 1$,
the local factor at $0$ of the $r$-th moment $L$-function for ${\mathcal G}_f$ is given by
$$
\det(1-\Frob_0T|([{\mathcal G}_f]^{r})^{I_0})=\prod_{\chi}(1-g(\chi,\psi)^rT)
$$
where the product is taken over all non-trivial characters $\chi$ of $k^\star$ such that
$\chi^e={\mathbf 1}$ for $e=\gcd(d,r)$ and $g(\chi,\psi)=-\sum_t\chi(t)\psi(t)$ is the corresponding Gauss sum.
\end{cor}

{\bf Proof.} As a representation of the inertia group $I_0$, ${\mathcal G}_f$ is the direct sum
$\bigoplus_{i=1}^{d-1}{\mathcal L}_{\chi^i}$ for a character $\chi$ of order $d$. So in the
Grothendieck group of $\QQ[I_0]$-modules we have
$[{\mathcal G}_f]^r=\bigoplus_{i=1}^{d-1}{\mathcal L}_{\chi^i}^{\otimes r}=\bigoplus_{i=1}^{d-1}{\mathcal L}_{\chi^{ir}}$.
For a given $i$, ${\mathcal L}_{\chi^{ir}}$ is trivial as a representation of $I_0$ if and only if $\chi^{ir}$
is trivial, that is, if and only if $ir$ is a multiple of $d$.

Writing $d=d'e$ with $e=\gcd(d,r)$. The trivial summands correspond to $i=d',2d',\ldots,(e-1)d'$.
The characters $\chi^i$ are then exactly the non-trivial characters of $k^\star$ whose $e$-th power is trivial,
and the corresponding Frobenius eigenvalues are $(\chi^i(a_d)g(\bar\chi^i,\psi))^r=
\chi^{ir}(a_d)g(\bar\chi^i,\psi)^r=g(\bar\chi^i,\psi)^r$ by proposition \ref{monodromy0}.\hfill$\Box$
\bigskip

\begin{cor} \label{localinf}\emph{(Local factor at $\infty$ of the moment $L$-function)}
Suppose that $p>2$, $f'$ is square-free and all its roots are in $k$. For any positive integer $r\geq 1$,
the local factor at $\infty$ of the $r$-th moment $L$-function for ${\mathcal G}_f$ is given by
$$
\det(1-\Frob_{\infty} T|([{\mathcal G}_f]^{r})^{I_{\infty}})=
$$
$$
=\left\{
\begin{array}{ll}
(1-(\rho(-1)q)^{r/2}T)^{d-1} & \mbox{if }2p|r  \\
(1-(\rho(-1)q)^{r/2}T)^{m}  & \mbox{if $2|r$, $(r,p)=1$ and $f$ has double roots} \\
1 & \mbox{otherwise}
\end{array}
\right.
$$
where $\rho:k^\star\to\{1,-1\}$ is the quadratic character and $m$ is the number of double roots of $f$.
\end{cor}

{\bf Proof.} As a representation of the inertia group $I_\infty$, ${\mathcal G}_f$ is the direct
sum $\bigoplus_{i=1}^{d-1}({\mathcal L}_{\psi_{s_i}}\otimes{\mathcal L}_\rho)$
where $\psi_{s_i}(t)=\psi(s_i t)$. So in the Grothendieck group of $\QQ[I_\infty]$-modules
we have $[{\mathcal G}_f]^r=\bigoplus_{i=1}^{d-1}({\mathcal L}_{\psi_{s_i}}^{\otimes r}
\otimes{\mathcal L}_\rho^{\otimes r})=\bigoplus_{i=1}^{d-1}({\mathcal L}_{\psi_{rs_i}}
\otimes{\mathcal L}_{\rho^r})$. The term $({\mathcal L}_{\psi_{rs_i}}\otimes{\mathcal L}_{\rho^r})$
is trivial if and only if $\rho^r$ and $\psi_{rs_i}$ are both trivial, that is, if and only if $r$ is even
and $rs_i=0$. That can only happen when either $r$ is divisible by $2p$ or $r$ is even and $s_i=0$.

In the first case the inertia group $I_\infty$ acts trivially on every term, and the Frobenius
eigenvalues are all equal to $(\pm g(\rho,\psi))^r=g(\rho,\psi)^r=(\rho(-1)q)^{r/2}$
by proposition \ref{monodromyinf}. In the second case, if $(r,p)=1$, the inertia group only
acts trivially on the $m$ terms for which $s_i=0$, and the corresponding Frobenius
eigenvalue is again $(\rho(-1)q)^{r/2}$.\hfill$\Box$

\bigskip

We now give some geometric conditions on $f$ that ensure that the trivial factors $P(T)$ and $P'(T)$ disappear:

\begin{prop}
 Suppose that $f'$ is square-free, and either:
\begin{enumerate}
 \item $r$ is odd, or
\item the hypersurface defined by $f(x_1)+\cdots+f(x_r)=0$ in $\AAA^r_k$ is non-singular. 
\end{enumerate}
Then $P(T)=P'(T)=1$.
\end{prop}

\begin{proof}
 We will check that, for every $i=0,\ldots,r$, the action of $\pi_1(\GG_{m,\bar k})$ on the sheaf $\Sym^{r-i}\GGG_f\otimes\wedge^i\GGG_f$ has no non-zero invariants. Since $\Sym^{r-i}\GGG_f\otimes\wedge^i\GGG_f$ is a subsheaf of $\bigotimes^r\GGG_f$ for every $i$, it suffices to prove it for the latter.

By Proposition \ref{monodromyinf}, the inertia group $I_\infty$ acts on $\GGG_f$ through the direct sum of the characters $\LL_\rho\otimes\LL_{\psi_{f(z)}}$ for every root $z$ of $f'$. Therefore it acts on its $r$-th tensor power as the direct sum of the characters $\LL_{\rho^r}\otimes\LL_{\psi_{f(z_1)+\cdots+f(z_r)}}$ for all $r$-tuples $(z_1,\ldots,z_r)$ of roots of $f'$. If $r$ is odd, none of these is the trivial character, since $\LL_{\rho^r}=\LL_{\rho}$ (which is totally and tamely ramified at infinity) can not be isomorphic to $\LL_{\psi_t}$ (which is either trivial or totally wild at infinity) for any $t$. If the hypersurface $f(x_1)+\cdots+f(x_r)=0$ is non-singular, the sums $f(z_1)+\cdots+f(z_r)$ are always non-zero, and therefore $\LL_{\rho^r}\otimes\LL_{\psi_{f(z_1)+\cdots+f(z_r)}}$ is totally wild at infinity.

In either case, $\Sym^{r-i}\GGG_f\otimes\wedge^i\GGG_f$ has no non-zero invariants under the action of $I_\infty$ and, a fortiori, under the action of the larger group $\pi_1(\GG_{m,\bar k})$.
\end{proof}

\begin{cor}\label{bound}
 Let $f\in k[x]$ be a polynomial of degree $d$ prime to $p>2$ and $r$ a positive integer. Suppose that $f'$ is square-free. If $r$ is even, suppose additionally that the hypersurface defined by $f(x_1)+\cdots+f(x_r)=0$ in $\AAA^r_k$ is non-singular. Then the number $N_r(f)$ of $k_r$-rational points on the curve
$$
y^q-y=f(x)
$$
satisfies the estimate
$$
|N_r(f)-q^r|\leq C_{d,r} q^{\frac{r+1}{2}}
$$
where
$$
C_{d,r}=\sum_{i=0}^r |i-1|{{d-2+r-i}\choose{r-i}}{{d-1}\choose{i}}
$$
is independent of $q$.
\end{cor}

\begin{proof}
 Under the hypotheses of the corollary, the previous result shows that $\pi_1(\GG_{m,\bar k})$ has no non-zero invariants on $\Sym^{r-i}\GGG_f\otimes\wedge^i\GGG_f$. Therefore, $\HH^2_c(\GG_{m,\bar k},\Sym^{r-i}\GGG_f\otimes\wedge^i\GGG_f)=0$, and formula \ref{lfunction} reduces to
$$
L^r(f,\psi,T)=\prod_{i=0}^r \det(1-\mathrm{Frob}_kT|\HH^1_c(\GG_{m,\bar k},\Sym^{r-i}\GGG_f\otimes\wedge^i\GGG_f))^{(-1)^{i-1}(i-1)}
$$

In particular, by \ref{nr},
$$
N_r(f)-q^r=\sum_{i=0}^r (-1)^{i-1}(i-1)\cdot\Trace(\Frob_k|\HH^1_c(\GG_{m,\bar k},\Sym^{r-i}\GGG_f\otimes\wedge^i\GGG_f)).
$$

Since $\HH^1_c(\GG_{m,\bar k},\Sym^{r-i}\GGG_f\otimes\wedge^i\GGG_f)$ is mixed of weight $\leq r+1$, we get the estimate
$$
|N_r(f)-q^r|\leq \left(\sum_{i=0}^r| i-1| \cdot\dim (\HH^1_c(\GG_{m,\bar k},\Sym^{r-i}\GGG_f\otimes\wedge^i\GGG_f))\right)\cdot q^{\frac{r+1}{2}}.
$$

Since $\HH^1_c$ is the only non-zero cohomology group of $\Sym^{r-i}\GGG_f\otimes\wedge^i\GGG_f$, we have
$$
\dim \,\HH^1_c(\GG_{m,\bar k},\Sym^{r-i}\GGG_f\otimes\wedge^i\GGG_f)=-\chi(\GG_{m,\bar k},\Sym^{r-i}\GGG_f\otimes\wedge^i\GGG_f)=
$$
$$
=\Swan_\infty(\Sym^{r-i}\GGG_f\otimes\wedge^i\GGG_f)
$$
by the Grothendieck-N\'eron-Ogg-Shafarevic formula, since $\Sym^{r-i}\GGG_f\otimes\wedge^i\GGG_f$ is tamely ramified at $0$. Now by \ref{monodromyinf}, all slopes at infinity of $\Sym^{r-i}\GGG_f\otimes\wedge^i\GGG_f$ are $0$ or $1$, so
$$
 \Swan_\infty(\Sym^{r-i}\GGG_f\otimes\wedge^i\GGG_f)\leq
$$
$$
\leq\mathrm{rank}(\Sym^{r-i}\GGG_f\otimes\wedge^i\GGG_f)={{d-2+r-i}\choose{r-i}}{{d-1}\choose i}.
$$
The proof is complete. 
\end{proof}

The non-singularity condition is generic on $f$ if $r$ is not a
multiple of $p$: in fact, there can be at most ${d+r-2}\choose r$
values of $\lambda\in\bar k$ for which $f(x)+\lambda$ does not
satisfy the condition. If $r$ is divisible by $p$, then
$f(x_1)+\cdots +f(x_r)=0$ always defines a singular affine
hypersurface and thus Theorem \ref{introduction} is empty if $r$ is further even.
In such cases, we can use the refinement in next section.

\section{Refinements using global monodromy.}\label{refinements}

In this section we will relax the hypotheses of Corollary \ref{bound} using Katz's computation of the global monodromy of $\GGG_f$. In particular, we will give conditions on $f$ that make the given bound hold for any $r$.

Let $G=\overline{\pi_1(\GG_{m,\bar k})}^{Zar}\subseteq GL(V)$ be the geometric monodromy group of $\GGG_f$, where $V$ is its generic stalk. Let $z_1,\ldots,z_{d-1}$ be the roots of $f'$ in $\bar k$, let $s_i=f(z_i)$ and $s=s_1+\cdots+s_{d-1}$.

\begin{prop}\label{SL}
 Suppose that $p>2d-1$ and the $(d-1)(d-2)$ numbers $s_i-s_j$ for $i\neq j$ are all distinct. Then $G$ is given by
$$
\left\{
\begin{array}{ll}
SL(V) & \mbox{if $d$ is odd and $s=0$} \\
GL_p(V) & \mbox{if $d$ is odd and $s\neq 0$} \\
GL_2(V)=\pm SL(V) & \mbox{if $d$ is even and $s=0$} \\
GL_{2p}(V) & \mbox{if $d$ is even and $s\neq 0$}
\end{array}
\right.
$$
where $GL_m(V)=\{A\in GL(V)|\det(A)^m=1\}$.
\end{prop}

\begin{proof}
 The hypothesis forces the $s_i$ to be distinct (otherwise $0$ would appear at least twice as a difference of two critical values). Since $p>d$, \cite[Lemma 7.10.2.3]{esde} shows that $\FF_f$ is a geometrically irreducible tame reflection sheaf. Then by \cite[Theorem 7.9.6]{esde}, $G$ must contain $SL(V)$. Since $SL(V)$ is connected, it must be contained in the unit connected component $G_0$ of $G$. On the other hand, since $\GGG_f$ is also geometrically irreducible (since Fourier transform preserves irreducibility), $G_0$ is a semisimple algebraic group \cite[Corollaire 1.3.9]{De} so it must be $SL(V)$. In order to determine $G$ completely, we only need to know the image of its determimant, but by Corollary \ref{determinant} we know it is trivial for $d$ odd and $s=0$ and the group of $p$-th roots (respectively square roots, $2p$-th roots) of unity for $d$ odd and $s\neq 0$ (resp. $d$ even and $s=0$, $d$ even and $s\neq 0$).
\end{proof}

\begin{cor}\label{estimateSL}
 Under the hypotheses of Proposition \ref{SL}, for any integer $r\geq 1$ the number $N_r(f)$ of $k_r$-rational points on the curve
$$
y^q-y=f(x)
$$
satisfies the estimate
$$
|N_r(f)-q^r|\leq C_{d,r} q^{\frac{r+1}{2}}
$$
where
$$
C_{d,r}=\sum_{i=0}^r |i-1|{{d-2+r-i}\choose{r-i}}{{d-1}\choose{i}},
$$
unless $d$ is odd, $s=0$ and $r=d-1$, in which case there exists $\beta=\pm 1$ such that $N_r(f)$ satisfies the estimate
$$
|N_r(f)-(q^r+\beta q^{\frac{r}{2}+1})|\leq C_{d,r} q^{\frac{r+1}{2}}.
$$
Moreover, if $k$ contains all $d$-th roots of unity then $\beta=1$.
\end{cor}

\begin{proof}
 For the first statement, we only need to show that $\pi_1(\GG_{m,\bar k})$ has no non-zero invariants on $\Sym^{r-i}\GGG_f\otimes\wedge^i\GGG_f$ for any $i$, the result follows exactly as in Corollary \ref{bound}. Equivalently, we need to show that $\Sym^{r-i}V\otimes\wedge^i V$ has no non-zero invariants under the action of $G$.

As a representation of $SL(V)$, we have
$$
\Sym^{r-i}V\otimes\wedge^i V=\Hom(\wedge^i
V^\star,\Sym^{r-i}V)=\Hom(\wedge^{d-1-i} V,\Sym^{r-i}V)
$$
whose invariant subspace, for $i\geq 0$, is $0$ except in the cases $r-i=d-1-i=0$ and $r-i=d-1-i=1$, where it is one-dimensional. In particular, $SL(V)$ (and, a fortiori, $G$) has no non-zero invariants on $\Sym^{r-i}V\otimes\wedge^i V$ for any $i\geq 0$ if $r\neq d-1$.

Suppose that $r=d-1$, and let $W_i$ be the one-dimensional subspace of $\Sym^{r-i}V\otimes\wedge^i V$ invariant under $SL(V)$, for $i=r-1$ or $i=r$. The factor group $G/SL(V)=\mu_m$ acts on $W_i$, where $m$ is given in the previous Proposition. Let $A=\mathrm{diag}(\zeta,\ldots,\zeta)\in G$ be a scalar matrix, where $\zeta\in\QQ$ is a primitive $m(d-1)$-th root of unity. Then the class of $A$ generates the cyclic group $G/SL(V)$, so $G$ fixes $W_i$ if and only if $A$ does. But $A$ acts on $W_i$ by multiplication by $\zeta^r$, so this action is trivial if and only if $\zeta^r=1$, that is, if and only if $m(d-1)$ divides $r=d-1$, which can only happen for $m=1$, that is, in the case where $d$ is odd and $s=0$.

It remains to prove the second estimate in this case. Since $G=SL(V)$, the determinant of $\GGG_f$ is geometrically trivial, so it is $(q^{\frac{d-1}{2}}\beta)^{\deg}$ for some $\beta$ with $|\beta|=1$. For $i=r$, $\Sym^{r-i}V\otimes\wedge^i V=\wedge^{d-1}V=\det V$ and therefore Frobenius acts
by multiplication by $q^{\frac{d-1}{2}}\beta$.
For $i=r-1$, $\Sym^{r-i}V\otimes\wedge^i V=V\otimes\wedge^{d-2}V=\Hom(V,V)\otimes\det V$
and the $G$-invariant part is again $\det V$, on which Frobenius acts by multiplication
by $q^{\frac{d-1}{2}}\beta$. We conclude that
$$
\prod_{i=0}^r \det(1-\mathrm{Frob}_kT|\HH^2_c(\GG_{m,\bar
k},\Sym^{r-i}\GGG_f\otimes\wedge^i\GGG_f))^{(-1)^{i-1}(i-1)}
$$
$$
=\det(1-\mathrm{Frob}_kT|(\det\GGG_f)(-1))^{(-1)^{r-2}(r-2) + (-1)^{r-1}(r-1)}
$$
$$
=(1-q^{\frac{r}{2}+1}\beta T)^{(-1)^{r-1}}=(1-q^{\frac{r}{2}+1}\beta T)^{-1}
$$
since $r-1=d-2$ is odd.

From equation (\ref{lfunction}) we then get that
$$
\frac{L^r(f,\psi,T)}{ 1-q^{\frac{r}{2}+1} \beta T}= {\prod_{i=0}^r \det(1-\mathrm{Frob}_kT|\HH^1_c(\GG_{m,\bar k},\Sym^{r-i}\GGG_f\otimes\wedge^i\GGG_f)^{(-1)^{i-1}(i-1)}}
$$
and, in particular, by (\ref{nr})
$$
N_r(f)-q^r-\beta q^{\frac{r}{2}+1} =\sum_{i=0}^r (-1)^{i-1}(i-1)\cdot\Trace(\Frob_k|\HH^1_c(\GG_{m,\bar k},\Sym^{r-i}\GGG_f\otimes\wedge^i\GGG_f)).
$$

But $L^r(f,\psi,T)$ has real coefficients (since taking complex conjugate is the same as replacing $f$ by $-f$ or, equivalently, taking the pull-back of $\GGG_f$ under the automorphism $t\mapsto -t$, so it gives the same $L^r$). Since $q^{\frac{r}{2}+1}\beta$ is its only reciprocal root of weight $r+2$, we conclude that $\beta=\pm 1$. Moreover, if $k$ contains all $d$-th roots of unity then $\beta=1$ by Corollary \ref{determinant}.

Using that $\Sym^{r-i}\GGG_f\otimes\wedge^i\GGG_f$ is pure of weight $r$, we obtain the estimate
$$
|N_r(f)-(q^r+\beta q^{\frac{r}{2}+1})|\leq \left(\sum_{i=0}^r |i-1|\cdot\dim (\HH^1_c(\GG_{m,\bar k},\Sym^{r-i}\GGG_f\otimes\wedge^i\GGG_f))\right)\cdot q^{\frac{r+1}{2}}
$$
We conclude as in corollary \ref{bound} using that, for the two values of $i$ for which $\HH^2_c(\GG_{m,\bar k},\Sym^{r-i}\GGG_f\otimes\wedge^i\GGG_f)$ is one-dimensional, the sheaf $\Sym^{r-i}\GGG_f\otimes\wedge^i\GGG_f$ has at least one slope equal to $0$ at infinity, and therefore
$$\dim\HH^1_c(\GG_{m,\bar k},\Sym^{r-i}\GGG_f\otimes\wedge^i\GGG_f)
$$
$$
=-\chi(\GG_{m,\bar k},\Sym^{r-i}\GGG_f\otimes\wedge^i\GGG_f)+\dim\HH^2_c(\GG_{m,\bar k},\Sym^{r-i}\GGG_f\otimes\wedge^i\GGG_f)
$$
$$
=\Swan_\infty(\Sym^{r-i}\GGG_f\otimes\wedge^i\GGG_f)+1\leq\mathrm{rank}(\Sym^{r-i}\GGG_f\otimes\wedge^i\GGG_f).$$
\end{proof}

The hypothesis of proposition \ref{SL} can easily be checked from the coefficients of $f$: Let $A_{f'}$ be the companion matrix of $f'$, and $B=f(A_{f'})$. The eigenvalues of the $(d-1)\times(d-1)$ matrix $B$ are $s_1,\ldots,s_{d-1}$, and its trace is $s$. Next we construct the $(d-1)^2\times(d-1)^2$ matrix $B\otimes I_{d-1}-I_{d-1}\otimes B$, whose eigenvalues are all differences $s_i-s_j$. Its characteristic polynomial is then of the form $T^{d-1}g(T)$. The hypothesis of proposition \ref{SL} are equivalent to the discriminant of $g(T)$ being non-zero.

\bigskip

We will now deal with an important class of polynomials to which \ref{SL} does not apply.

\begin{defn}
We say that a polynomial $f\in k[x]$ is \emph{quasi-odd} if there exist $a,b\in k$ such that $f(a-x)=b-f(x)$. In this case, the degree $d$ is necessarily odd. 
\end{defn}

Notice that $a$ and $b$ are then uniquely determined: if $f=c_dx^d+\cdots+c_1x+c_0$, $a=\frac{-2c_{d-1}}{dc_d}$ and $b=2f(\frac{a}{2})$. If $f$ is quasi-odd, the set of critical values of the map $f:\AAA^1_k\to\AAA^1_k$ is invariant under the involution $s\mapsto b-s$. In particular, their sum is $\frac{b(d-1)}{2}$.

\begin{lem}
 If there is $a\in k$ such that $f(a-x)=-f(x)$, the Tate-twisted sheaf $\GGG_f(1/2)$ on $\GG_{m,k}$ is self-dual.
\end{lem}

\begin{proof}
 Since $\GGG_{f(x-c)}\cong\GGG_{f(x)}$ for any $c\in k$, we may assume that $f$ is odd. The automorphism $x\mapsto -x$ induces an isomorphism $\FF_f\cong [-1]^\star\FF_f=\FF_{-f}$. Taking Fourier transform, we get an isomorphism $\GGG_f\cong\GGG_{-f}$. Composing with the duality pairing (cf. section 3) $\GGG_f\times\GGG_{-f}\to\QQ(-1)$ we get a perfect pairing $\GGG_f\times\GGG_f\to\QQ(-1)$ or, equivalently, $\GGG_f(1/2)\times\GGG_f(1/2)\to\QQ$.
\end{proof}

\begin{prop}\label{Sp}
 Let $f\in k[x]$ be quasi-odd. Label the critical values $s_i$ so that $s_{d-i}=b-s_i$ for $i=1,\ldots,d-1$. Suppose that $p>2d-1$ and the only equalities among the numbers $s_i-s_j$ for $i\neq j$ are $s_i-s_j=s_{d-j}-s_{d-i}$. Then $G=Sp(V)$ if $b=0$ (if and only if
 $s=0$, since $p>d-1$), and $G=\mu_p\cdot Sp(V)$ if $b\neq 0$.
\end{prop}

\begin{proof}
The hypothesis forces the $s_i$ to be distinct: if $s_i=s_j$ for $i\neq j$ then $s_i-s_j=s_j-s_i$, so $i=d-i$ and $j=d-j$, which is impossible since $d$ is odd. Then by \cite[Lemma 7.10.2.3]{esde} $\FF_f$ is a geometrically irreducible tame reflection sheaf. If $b=0$, we may assume as in the previous lemma that $f$ is odd. The self-duality of $\GGG_f(1/2)$ is symplectic (it suffices to show it geometrically, and that is done in \cite[Lemma 7.10.4]{esde}), so we have $G\subseteq Sp(V)$. We now apply \cite[Theorem 7.9.7]{esde}, from which $G$ must contain $SL(V)$, $Sp(V)$ or $SO(V)$, and therefore we must have $G=Sp(V)$.

If $b\neq 0$, $f(x)-\frac{b}{2}$ is quasi-odd with $b=0$, and $\GGG_f=\GGG_{f-b/2}\otimes\LL_{\psi_{b/2}}$. Let $H\subseteq\pi_1(\GG_{m,\bar k})$ be the kernel of the character $\LL_{\psi_{b/2}}$, it is an open normal subgroup of index $p$ and the restrictions of the representations $\GGG_f$ and $\GGG_{f-b/2}$ to $H$ are isomorphic. Since the monodromy group of $\GGG_{f-b/2}$ is $Sp(V)$, which does not have open subgroups of finite index, the closure of the image of $H$ on $GL(V)$ under $\GGG_f$ is the whole $Sp(V)$. Therefore, $Sp(V)\subseteq G$ and $G\subseteq \mu_p\cdot Sp(V)$, since $\pi_1(\GG_{m,\bar k})$ acts via $\LL_{\psi_{b/2}}$ by multiplication by $p$-th roots of unity. Since the determinant of $G$ is non-trivial by Corollary \ref{determinant}, it must be $\mu_p\cdot Sp(V)$.
\end{proof}

\begin{cor}\label{estimateSp}
 Under the hypotheses of Proposition \ref{Sp}, for any integer $r\geq 1$ the number $N_r(f)$ of $k_r$-rational points on the curve
$$
y^q-y=f(x)
$$
satisfies the estimate
$$
|N_r(f)-q^r|\leq C_{d,r} q^{\frac{r+1}{2}}
$$
where
$$
C_{d,r}=\sum_{i=0}^r |i-1|{{d-2+r-i}\choose{r-i}}{{d-1}\choose{i}},
$$
unless $r\leq d-1$ is even and either $b=0$ or $p$ divides $r$, in which case it satisfies the estimate
$$
|N_r(f)-(q^r+q^{\frac{r}{2}+1})|\leq C_{d,r} q^{\frac{r+1}{2}}.
$$
\end{cor}

\begin{proof}
 As a representation of $Sp(V)$, we have
$$
\Sym^{r-i}V\otimes\wedge^i V=\Hom(\wedge^i V,\Sym^{r-i}V)
$$
whose invariant subspace, by \cite[lemma on p.62]{KatzFS}, is $0$ except when $i$ is odd, $r=i+1$ and $i\leq d-1$, or when $i$ is even, $r=i$ and $i\leq d-1$. In particular, since $d$ is odd, $G$ has no non-zero invariants on $\Sym^{r-i}V\otimes\wedge^i V$ for any $i$ if $r$ is odd or $r>d-1$.

Suppose from now on that $r\leq d-1$ is even, and
let $W_i$ be the one-dimensional subspace of $\Sym^{r-i}V\otimes\wedge^i V$
invariant under $Sp(V)$, for $i=r-1$ or $i=r$. Consider the case where $b=0$ first.
Since $\GGG_f(1/2)$ is self-dual, all Frobenius images are in $Sp(V)=G$. In particular,
all Frobenii act trivially on $W_i(r/2)$, and therefore they act by multiplication
by $q^{\frac{r}{2}}$ on $W_i\subseteq\bigotimes^r V$. Therefore

$$
\prod_{i=0}^r \det(1-\mathrm{Frob}_kT|\HH^2_c(\GG_{m,\bar
k},\Sym^{r-i}\GGG_f\otimes\wedge^i\GGG_f))^{(-1)^{i-1}(i-1)}
$$
$$
=\det(1-\mathrm{Frob}_kT|W_{r-1}(-1))^{(-1)^{r-2}(r-2)}\det(1-\mathrm{Frob}_kT|W_r(-1))^{(-1)^{r-1}(r-1)}
$$
$$
=(1-q^{\frac{r}{2}+1}T)^{(-1)^{r-2}(r-2) + (-1)^{r-1}(r-1)}
$$
$$
=(1-q^{\frac{r}{2}+1}T)^{(-1)^{r-1}}=(1-q^{\frac{r}{2}+1}T)^{-1}
$$
since $r-1=d-2$ is odd.

In the case where $b\neq 0$, $G/Sp(V)\cong\mu_p$ acts on $W_i$.
Let
$$A=\mathrm{diag}(\zeta_p,\ldots,\zeta_p)\in G$$
be a scalar matrix, where $\zeta_p\in\QQ$ is a $p$-th root of unity.
Then the class of $A$ generates $G/Sp(V)$, so $G$ fixes $W_i$ if and only if $A$ does. But $A$ acts on $W_i$ by multiplication by $\zeta_p^r$, so this action is trivial if and only if $\zeta_p^r=1$, that is, if and only if $p$ divides $r$. In that case, $\Sym^{r-i}\GGG_f\otimes\wedge^i\GGG_f=(\Sym^{r-i}\GGG_{f-b/2}\otimes\wedge^i\GGG_{f-b/2})\otimes\LL_{\psi_{b/2}}^{\otimes r}=\Sym^{r-i}\GGG_{f-b/2}\otimes\wedge^i\GGG_{f-b/2}$, so we can apply the $b=0$ case and we get again
$$
\prod_{i=0}^r \det(1-\mathrm{Frob}_kT|\HH^2_c(\GG_{m,\bar k},\Sym^{r-i}\GGG_f\otimes\wedge^i\GGG_f))^{(-1)^{i-1}(i-1)}=
$$
$$
=(1-q^{\frac{r}{2}+1}T)^{(-1)^{r-1}}=(1-q^{\frac{r}{2}+1}T)^{-1}.
$$

We conclude as in corollary \ref{estimateSL}.\end{proof}

Again, the hypothesis of proposition \ref{Sp} can be checked from the coefficients of $f$: After adding a constant, we may assume that $b=0$. Let $A_{f'}$ be the companion matrix of $f'$, and $B=f(A_{f'})$. The eigenvalues of the $(d-1)\times(d-1)$ matrix $B$ are $s_1,\ldots,s_{d-1}$, and its trace is $s=\frac{b(d-1)}{2}$. Construct the $(d-1)^2\times(d-1)^2$ matrix $B\otimes I_{d-1}-I_{d-1}\otimes B$, whose eigenvalues are all differences $s_i-s_j$. Its characteristic polynomial is then of the form $T^{d-1}h(T/2)g(T)^2$, where $h(T)$ is the characteristic polynomial of $B$, since all non-zero roots different from $s_i-s_{d-i}=2s_i$ for $i=1,\ldots,d-1$ appear in pairs. The hypothesis of proposition \ref{Sp} is equivalent to the discriminant of $h(T/2)g(T)$ being non-zero.

\section{Generalization to Artin-Schreier hypersurfaces}

In this section we will extend corollary \ref{bound} to higher dimensional hypersurfaces. Since the proofs are very similar, we will only sketch them, 
indicating the differences where necessary.

Let $f\in k[x_1,\ldots,x_n]$ be a polynomial of degree $d$ prime to $p$, $C_f$ the Artin-Schreier hypersurface defined on $\AAA^{n+1}_k$ by the equation
\begin{equation}\label{hypersurface}
 y^q-y=f(x_1,\ldots,x_n).
\end{equation}
Denote by $N_r(f)$ its number of rational points over $k_r$. We have
again a formula
\begin{equation}\label{characterhyper}
 N_r(f)-q^{nr}=\sum_{t\in k^\star}\sum_{x\in k_r^n}\psi(t\cdot\Tr(f(x)))=\sum_{t\in k^\star}\sum_{x\in k_r^n}\psi(\Tr(tf(x)))
\end{equation}
where $\Tr$ denotes the trace map $k_r\to k$. 
Assume that $f$ is a Deligne polynomial, that is, the leading form of $f$ defines a smooth projective hypersurface 
of degree $d$ not divisible by $p$. Applying Deligne's bound \cite{De} to the above inner sum, one deduces that 
$$|N_r(f) - q^{nr}| \leq (q-1)(d-1)^n q^{\frac{nr}{2}}.$$
This is precisely Weil's bound in the case $n=1$. Our purpose of this section is to improve the above bound and 
obtain the estimate of the following form
$$|N_r(f)-q^{nr}| \leq C_{d,r}q^{\frac{nr+1}{2}},$$
for some constant $C_{d,r}$ depending only on $d, r$ and $n$.

Define $K_f=\R\pi_!{\mathcal L}_{\psi(tf(x))}\in{{\mathcal
D}^b_c({\mathbb G}_{m,k},\QQ)}$, where $\pi:{\mathbb
G}_m\times{\mathbb A}^n\to{\mathbb G}_m$ is the projection. The
trace formula implies that the trace of the action of the $r$-th power of a local Frobenius element at $t\in k^\star$ on $K_f$ is given by 
$\sum_{x\in k_r^n}\psi(\Tr(tf(x)))$. Suppose from now on that the homogeneous part $f_d$ of highest degree of $f$ defines a non-singular hypersurface. Then by \cite[3.7]{De}, $K_f$ is a single smooth sheaf $\GGG_f$ placed in degree $n$, of rank $(d-1)^n$ and pure of weight $n$. Therefore
$$
N_r(f)-q^{nr}= (-1)^n\sum_{t\in k^\star}\Tr({\rm Frob}^r_t|(\GGG_f)_t) =  (-1)^n\sum_{t\in k^\star}{\rm Tr}({\rm Frob}_t | [\GGG_f]^r_t)
$$
where
$$
[\GGG_f]^r =\sum_{i=0}^r (-1)^{i-1} (i -1)\cdot {\rm Sym}^{r-i}\GGG_f \otimes \wedge^i
\GGG_f
$$
is the $r$-th Adams operation on $\GGG_f$.

We can give an interpretation of $\GGG_f$ in terms of the Fourier transform like we did in the one-dimensional case. Exactly as in lemma \ref{fourier}, we can show
\begin{lem}
 The object $\GGG_f[1]\in{\mathcal D}^b_c(\GG_m,\QQ)$ is the restriction to $\GG_m$ of the Fourier transform of $\R f_!\QQ[n]$ with respect to $\psi$.
\end{lem}

We compactify $f$ via the map $\tilde f:X\to\AAA^1_k$, where $X\subseteq \PP^n\times\AAA^1$ is defined by the equation $F(x_0,x_1,\ldots,x_n)=tx_0^d$, $F$ being the homogenization of $f$ with respect to the variable $x_0$, and $\tilde f$ the restiction of the second projection to $X$. Suppose that the subscheme of $\AAA^n_k$ defined by the ideal $\langle\partial f/\partial x_1,\ldots,\partial f/\partial x_n\rangle$ is finite \'etale over $k$, and the images of its $\bar k$-points under $f$ are distinct. Then for every $s\in\bar k$, the fibre $X_s$ has at worst one isolated non-degenerate quadratic singularity, which is located on the affine part (since the part at infinity is defined for every fibre by $f_d(x)=0$ and is therefore non-singular).

 We have a distinguished triangle
$$
\R f_!\QQ\to\R \tilde f_\star\QQ\to \R (\tilde f_{|X_0})_\star\QQ\to
$$
where $X_0=X\backslash\AAA^n\cong Y\times\AAA^1$, $Y$ being the smooth hypersurface defined in $\PP^{n-1}$ by $f_d=0$. Since $\R (\tilde f_{|X_0})_\star\QQ$ is just the constant object $\R\Gamma(Y,\QQ)$, its Fourier transform is supported at $0$. So
$$
\GGG_f[1]\cong(FT_\psi\R f_!\QQ[n])_{|\GG_{m,k}}\cong(FT_\psi\R \tilde f_\star \QQ[n])_{|\GG_{m,k}}.
$$

\begin{prop}\label{monodromyinfn}
 Suppose $p>2$. Under the previous hypotheses, let $z_1,\ldots,z_{(d-1)^n}\in \AAA^n_{\bar k}$ be the distinct points such that $\frac{\partial f}{\partial x_i}(z_j)=0$ for all $i=1,\ldots,n$, and let $s_i=f(z_i)$. The action of the inertia group $I_\infty$ at infinity on $\GGG_f$ decomposes as a direct sum $\bigoplus{\LL_{\psi_{s_i}}}$ if $n$ is even, and $\bigoplus(\LL_{\rho}\otimes\LL_{\psi_{s_i}})$ if $n$ is odd, where $\rho$ is the unique character of $I_\infty$ of order $2$.
\end{prop}

\begin{proof}
 We will obtain, for every $i$, a factor $\LL_{\psi_{s_i}}$ (resp. $\LL_\rho\otimes\LL_{\psi_{s_i}}$) in the local monodromy of $\GGG_f$ at infinity. Since the rank is $(d-1)^n$ and these characters are pairwise non-isomorphic, this will determine the action of $I_\infty$ completely.

 Let $S=\{s_i|i=1,\ldots,(d-1)^n\}$, and $U=\AAA^1\backslash S$. Since $\tilde f$ is proper and smooth over $U$, $\R^i\tilde f_\star\QQ$ is smooth on $U$ for every $i$. Since $X_s$ contains one isolated non-degenerate quadratic singularity for each $s\in S$, by \cite[4.4]{weil1} the sheaves $\R^i\tilde f_\star\QQ$ are smooth on $\AAA^1$ for $i\neq n-1,n$. In particular, their Fourier transforms are supported at $0$. We conclude that there is a distinguished triangle
$$
(FT_\psi\R^{n-1}\tilde f_\star\QQ[1])_{|\GG_{m,k}}\to\GGG_f[1]\to(FT_\psi\R^{n}\tilde f_\star\QQ[0])_{|\GG_{m,k}}\to
$$
and therefore an exact sequence of sheaves
\begin{equation}\label{preseq}
0\to\HHH^{-1}(FT_\psi\R^{n-1}\tilde f_\star\QQ[1])_{|\GG_{m,k}}\to\GGG_f\to\HHH^{-1}(FT_\psi\R^{n}\tilde f_\star\QQ[0])_{|\GG_{m,k}}\to
\end{equation}
$$
\to \HHH^0(FT_\psi\R^{n-1}\tilde f_\star\QQ[1])_{|\GG_{m,k}}\to 0
$$
since $FT_\psi\R^{n}\tilde f_\star\QQ[0]$ can only have non-zero cohomology sheaves in degrees $1$, $0$ and $-1$. Furthermore $\HHH^0(FT_\psi\R^{n-1}\tilde f_\star\QQ[1])$ is punctual, so this induces an exact sequence of $I_\infty$-representations
\begin{equation}\label{seqrep}
0\to\HHH^{-1}(FT_\psi\R^{n-1}\tilde f_\star\QQ[1])\to\GGG_f\to\HHH^{-1}(FT_\psi\R^{n}\tilde f_\star\QQ[0])\to 0.
\end{equation}

Let $V$ be the generic stalk of $\R^{n-1}\tilde f_\star\QQ$. Suppose that $n$ is odd, and let $s\in S$. Then by \cite[4.3 and 4.4]{weil1}, the inertia group $I_s$ acts on $V$ with invariant space $V_{I_s}$ of codimension $1$ (the orthogonal complement of the 'vanishing cycle' $\delta$) and on the quotient $V/V_{I_s}$ via its quadratic character $\rho$. Moreover, $\R^{n-1}\tilde f_\star\QQ$ is isomorphic at $s$ to the extension by direct image of its restriction to the generic point. By Laumon's local Fourier transform \cite[Section 7.4]{esde}, the action of the inertia group $I_\infty$ on $\HHH^{-1}(FT_\psi\R^{n-1}\tilde f_\star\QQ[1])$ (and thus on $\GGG_f$ by (\ref{seqrep})) contains a subcharacter isomorphic to $\LL_\rho\otimes\LL_{\psi_s}$. 

Suppose now that $n$ is even, and let $s\in S$. By \cite[4.3 and 4.4]{weil1}, there are two possibilities: if the 'vanishing cycle' $\delta$ is non-zero, the inertia group $I_s$ acts on $V$ with invariant space $V_{I_s}$ of codimension $1$ (the orthogonal complement of $\delta$) and trivially on the quotient $V/V_{I_s}$. Moreover, $\R^{n-1}\tilde f_\star\QQ$ is isomorphic at $s$ to the extension by direct image of its restriction to the generic point. By Laumon's local Fourier transform \cite[Section 7.4]{esde}, the action of the inertia group $I_\infty$ on $\HHH^{-1}(FT_\psi\R^{n-1}\tilde f_\star\QQ[1])$ (and thus on $\GGG_f$ by (\ref{seqrep})) contains a subcharacter isomorphic to $\LL_{\psi_s}$. 

 If $\delta=0$, then $I_s$ acts trivially on $V$, and there is an exact sequence of sheaves:
$$
0\to(\QQ)_s\to\R^n\tilde f_\star\QQ\to j_{s\star} {j_s^\star}\R^n\tilde f_\star\QQ\to 0
$$
where $(\QQ)_s$ is the punctual object $\QQ$ supported on $s$ and $j_s:\AAA^1-\{s\}\hookrightarrow\AAA^1$ is the inclusion. Taking Fourier transform, we deduce a distinguished triangle
$$
\LL_{\psi_s}[1]\to FT_\psi\R^n\tilde f_\star\QQ[0]\to FT_\psi j_{s\star} {j_s^\star}\R^n\tilde f_\star\QQ[0]\to
$$
and in particular an injection
$$
0\to \LL_{\psi_s}\to \HHH^{-1}(FT_\psi\R^n\tilde f_\star\QQ[0]).
$$
By (\ref{seqrep}), this gives a subcharacter isomorphic to $\LL_{\psi_s}$ in the monodromy of $\GGG_f$ at infinity.
\end{proof}

For completeness, we determine also the monodromy of $\GGG_f$ at $0$.

\begin{prop}
 The inertia group $I_0$ at $0$ acts on $\GGG_f$ as a direct sum $\bigoplus n_\chi\LL_\chi$ where the sum is taken over all characters $\chi$ of $I_0$ such that $\chi^d$ is trivial, $n_\chi=\frac{1}{d}((d-1)^n-(-1)^n)$ if $\chi$ is non-trivial and $n_\chi=(-1)^n+\frac{1}{d}((d-1)^n-(-1)^n)$ if $\chi$ is trivial.
\end{prop}

\begin{proof}
 We will show that, for every $\chi$, the action of $I_0$ on $\GGG_f$ contains $n_\chi$ Jordan blocks for the character $\chi$. Since these numbers add up to $(d-1)^n$, which is the dimension of the representation $\GGG_f$, this will prove that the action is semisimple and determine it completely.

Let $\chi$ be non-trivial such that $\chi^d={\mathbf 1}$. Since adding a constant $a$ to $f$ corresponds to tensoring $\GGG_f$ with the Artin-Schreier sheaf $\LL_{\psi_a}$ and this does not change the monodromy at $0$, we can assume that $\GGG_f$ is totally wild at $\infty$ (or equivalently, that the hypersurface $f(x)=0$ is non-singular). Then so is $\GGG_f\otimes\LL_{\bar\chi}$. The number of Jordan blocks associated of $\LL_\chi$ in the representation of $I_0$ given by $\GGG_f$ is the dimension of the $I_0$-invariant subspace of $\GGG_f\otimes\LL_{\bar\chi}$. If $j:\GG_{m,\bar k}\to\AAA^1_{\bar k}$ and $i:\{0\}\to\AAA^1_{\bar k}$ are the inclusions, we have an exact sequence
$$
0\to j_!(\GGG_f\otimes\LL_{\bar\chi})\to j_\star (\GGG_f\otimes\LL_{\bar\chi})\to i_\star i^\star j_\star(\GGG_f\otimes\LL_{\bar\chi})\to 0
$$
and therefore
$$
0\to(\GGG_f\otimes\LL_{\bar\chi})^{I_0}\to\HH^1_c(\GG_{m,\bar k},\GGG_f\otimes\LL_{\bar\chi})\to\HH^1_c(\AAA^1_{\bar k},j_\star(\GGG_f\otimes\LL_{\bar\chi}))\to 0.
$$
Since $\GGG_f\otimes\LL_{\bar\chi}$ is totally wild at $\infty$, the latter cohomology group is pure of weight $n+1$. So the dimension of $(\GGG_f\otimes\LL_{\bar\chi})^{I_0}$ is the dimension of the weight $\leq n$ part of $\HH^1_c(\GG_{m,\bar k},\GGG_f\otimes\LL_{\bar\chi})$.

By the projection formula, $$\GGG_f\otimes\LL_{\bar\chi}=(\R^n\pi_!\LL_{\psi(tf(x))})\otimes\LL_{\bar\chi}\cong\R^n\pi_!(\LL_{\psi(tf(x))}\otimes\LL_{\bar\chi(t)}),$$
so
$$\HH^1_c(\GG_{m,\bar k},\GGG_f\otimes\LL_{\bar\chi})=\HH^{n+1}_c(\GG_{m,\bar k}\times\AAA^n_{\bar k},\LL_{\psi(tf(x))}\otimes\LL_{\bar\chi(t)})
$$
since $\R^i\pi_!\LL_{\psi(tf(x))}=0$ for $i\neq n$. Let $Z\subset\AAA^n_k$ be the closed subset defined by $f(x)=0$ and $U$ its open complement. The sheaf $\LL_{\psi(tf(x))}$ is trivial on $\GG_m\times Z$, so $\HH^\star_c(\GG_{m,\bar k}\times Z,\LL_{\psi(tf(x))}\otimes\LL_{\bar\chi(t)})=\HH^\star_c(\GG_{m,\bar k}\times Z,\LL_{\bar\chi(t)})=\HH^\star_c(\GG_{m,\bar k},\LL_{\bar\chi})\otimes\HH^\star_c(Z\otimes\bar k,\QQ)=0$ since $\chi$ is non-trivial. By excision we get an isomorphism $\HH^{n+1}_c(\GG_{m,\bar k}\times\AAA^n_{\bar k},\LL_{\psi(tf(x))}\otimes\LL_{\bar\chi(t)})\cong\HH^{n+1}_c(\GG_{m,\bar k}\times U,\LL_{\psi(tf(x))}\otimes\LL_{\bar\chi(t)})$.

Consider the automorphism $\phi:\GG_m\times U\to\GG_m\times U$ given by $\phi(t,x)=(tf(x),x)$. Then $\phi_\star(\LL_{\psi(tf(x))}\otimes\LL_{\bar\chi(t)})=\LL_{\psi(t)}\otimes\LL_{\bar\chi(t/f(x))}=\LL_{\psi(t)}\otimes\LL_{\bar\chi(t)}\otimes\LL_{\chi(f(x))}$. So
$$
\HH^{n+1}_c(\GG_{m,\bar k}\times U,\LL_{\psi(tf(x))}\otimes\LL_{\bar\chi(t)})\cong\HH^{n+1}_c(\GG_{m,\bar k}\times U,\LL_{\psi(t)}\otimes\LL_{\bar\chi(t)}\otimes\LL_{\chi(f(x))})
$$
which, by K\"unneth, is isomorphic to $\HH^1_c(\GG_{m,\bar k},\LL_\psi\otimes\LL_{\bar\chi})\otimes\HH^n_c(U\otimes\bar k,\LL_{\chi(f)})$ (since $\HH^i_c(\GG_{m,\bar k},\LL_\psi\otimes\LL_{\bar\chi})=0$ for $i\neq 1$). The first factor is one-dimensional and pure of weight $1$, so we want the dimension of the weight $\leq n-1$ part of $\HH^n_c(U\otimes\bar k,\LL_{\chi(f)})$. By \cite[Theorem 2.2]{mult}, this dimension is $n_\chi=\frac{1}{d}((d-1)^n-(-1)^n)$.

Similarly, if $\chi={\mathbf 1}$ is the trivial character, the searched dimension is the dimension of the weight $\leq n$ part of $\HH^{n+1}_c(\GG_{m,\bar k}\times\AAA^n_{\bar k},\LL_{\psi(tf(x))})$. From the exact sequence
$$
\ldots\to\HH^n_c(\{0\}\times\AAA^n_{\bar k},\QQ)\to\HH^{n+1}_c(\GG_{m,\bar k}\times\AAA^n_{\bar k},\LL_{\psi(tf(x))})\to
$$
$$
\to\HH^{n+1}_c(\AAA^1_{\bar k}\times\AAA^n_{\bar k},\LL_{\psi(tf(x))})\to\HH^{n+1}_c(\{0\}\times\AAA^n_{\bar k},\QQ)\to\ldots
$$
we get an isomorphism $\HH^{n+1}_c(\GG_{m,\bar k}\times\AAA^n_{\bar k},\LL_{\psi(tf(x))})\cong\HH^{n+1}_c(\AAA^1_{\bar k}\times\AAA^n_{\bar k},\LL_{\psi(tf(x))})$. Now let $\pi:\AAA^1\times\AAA^n\to\AAA^n$ be the projection, by the base change theorem we have $\R^2\pi_!\LL_{\psi(tf(x))}=i_\star\QQ(-1)$, where $i:Z\to\AAA^n$ is the inclusion of the closed set where $f(x)=0$, and $\R^i\pi_!\LL_{\psi(tf(x))}=0$ for $i\neq 2$. So we need the dimension of the weight $\leq n-2$ part of $\HH^{n-1}_c(Z,\QQ)$. Let $\overline Z$ be the projective closure of $Z$ and $Z_0={\overline Z}\backslash Z$, we have an exact sequence
$$
\ldots\to\HH^{n-2}(\overline Z,\QQ)\to\HH^{n-2}(Z_0,\QQ)\to\HH^{n-1}_c(Z,\QQ)\to\HH^{n-1}(\overline Z,\QQ)\to\ldots
$$
Since $\overline Z$ is smooth, $\HH^{n-1}(\overline Z,\QQ)$ is pure of weight $n-1$, and therefore the weight $\leq n-2$ part of $\HH^{n-1}_c(Z,\QQ)$ is the cokernel of the map $\HH^{n-2}(\overline Z,\QQ)\to\HH^{n-2}(Z_0,\QQ)$, that is, the primitive part $\mathrm{Prim}^{n-2}(Z_0,\QQ)$ of the middle cohomology group of $Z_0$, which has dimension $n_{\mathbf 1}=(-1)^n+\frac{1}{d}((d-1)^n-(-1)^n)$.
\end{proof}

\begin{cor}\label{determinantn}
 Let $s=\sum_{i=1}^{(d-1)^n}s_i$. Over $\bar k$, the determinant of $\GGG_f$ is the Artin-Schreier sheaf $\LL_{\psi_s}$ if $n(d-1)$ is even, and the product $\LL_\rho\otimes\LL_{\psi_s}$ if $n(d-1)$ is odd.
\end{cor}

\begin{proof}
 The determinant is a smooth sheaf on $\GG_m$ of rank $1$. At $0$, its monodromy is the product of $\chi^{n_\chi}$ for all characters $\chi$ of $I_0$ such that $\chi^d$ is trivial. Since the non-trivial characters (except for the quadratic one) appear in conjugate pairs, the product is trivial if $d$ is odd, and comes down to $\rho^{n_\rho}$, which is $\rho$ or ${\mathbf 1}$ depending on the parity of $n_\rho=\frac{1}{d}((d-1)^n-(-1)^n)$, which is congruent to $n$ mod $2$, if $d$ is even.

At infinity, its monodromy is the product of the $\LL_{\psi_{s_i}}$ (resp. of the $\LL_\rho\otimes\LL_{\psi_{s_i}}$) if $n$ is even (resp. if $n$ is odd), which is $\LL_{\psi_s}$ (resp. $\LL_\rho\otimes\LL_{\psi_s}$) if $n(d-1)$ is even (resp. if $n(d-1)$ is odd). We conclude as in Corollary \ref{determinant}.
\end{proof}

We now give the higher dimensional analogue of Corollary \ref{bound}:

\begin{cor}
 Let $f\in k[x_1,\ldots,x_n]$ be a polynomial of degree $d$ prime to $p$ and $r$ a positive integer. Suppose that $p>2$, the highest degree homogeneous part of $f$ defines a non-singular hypersurface, the subscheme of $\AAA^n_k$ defined by the ideal $\langle\partial f/\partial x_1,\ldots,\partial f/\partial x_n\rangle$ is finite \'etale over $k$ and the images of its $\bar k$-points under $f$ are distinct. If $nr$ is even, suppose additionally that the hypersurface defined by $f(x_{1,1},\ldots,x_{1,n})+\cdots+f(x_{r,1},\ldots,x_{r,n})=0$ in $\AAA^{nr}_k=\Spec k[x_{i,j}|1\leq i\leq r,1\leq j\leq n]$ is non-singular. Then the number $N_r(f)$ of $k_r$-rational points on the hypersurface
$$
y^q-y=f(x_1,\ldots,x_n)
$$
satisfies the estimate
$$
|N_r(f)-q^{nr}|\leq C_{d,r} q^{\frac{nr+1}{2}}
$$
where
$$
C_{d,r}=\sum_{i=0}^r |i-1| {{(d-1)^n+r-i-1}\choose{r-i}}{{(d-1)^n}\choose{i}}
$$
is independent of $q$.
\end{cor}

The proof is identical to the one of Corollary \ref{bound}, using Proposition \ref{monodromyinfn}. In the $n$ even case we need the non-singularity hypothesis for any $r$, since the Kummer factor does not appear in the monodromy at infinity.

\end{document}